\documentclass[12pt]{article}

\usepackage{amsthm,amsmath,amssymb,enumerate}

\usepackage{graphicx}

\theoremstyle{plain}
\newtheorem{theorem}{Theorem}
\newtheorem{lemma}[theorem]{Lemma}

\theoremstyle{definition}

\theoremstyle{remark}

\title{\bf The 1/3-2/3 Conjecture for $N$-free ordered sets}

\author{Imed Zaguia\\
\small Dept of Mathematics \& Computer Science, Royal Military College of Canada\\
\small P.O.Box 17000, Station Forces, K7K 7B4 Kingston, Ontario CANADA \\[-0.8ex]
\small\tt imed.zaguia@rmc.ca}

\begin{document}
\maketitle


\begin{abstract}
A balanced pair in an ordered set
$P=(V,\leq)$ is a pair $(x,y)$ of elements of $V$ such that the
proportion of linear extensions of $P$ that put $x$ before $y$ is
in the real interval $[1/3, 2/3]$. We prove that every finite
$N$-free ordered set which is not totally ordered has a balanced
pair.

\end{abstract}

\input{epsf}

\section{Introduction}

Throughout, $P =(V, \leq)$ denotes a \emph{finite ordered set}, that
is, a finite set $V$ and a binary relation $\leq$ on $V$ which is
reflexive, antisymmetric and transitive. A \emph{linear extension} of $P =(V, \leq)$ is a linear ordering $\preceq$ of $V$
which extends $\leq$, i.e. such that $x\preceq y$ whenever $x \leq y$.

Suppose an unknown linear extension
$L$ of $P$ is to be determined using only comparisons between pairs of elements. At each step we ask a question of the form
"is it true that $x\prec y$?". We will get the answer before we can ask another question.
How many comparisons do we need to perform (in the worst case) in
order to determine $L$ completely? This is known as the problem of \emph{comparison sorting}.

Suppose that at each step we can find a pair $(x,y)$ of incomparable elements such that
the proportion of linear extensions of $P$ that put $x$ before $y$, denoted $\mathbb{P}(x\prec y)$, equals $\frac{1}{2}$. Then we need at least
$\log_{2}(e(P))$ comparisons where $e(P)$ denotes the number of linear extensions of $P$. This is not always
possible as shown by the example (i) depicted in Figure \ref{alpha}. Indeed, in that example the only possible values for
$\mathbb{P}(x\prec y)$ are 1/3 or 2/3.

Call a pair $(x,y)$ of elements of $V$ a \emph{balanced pair} in $P=(V,\leq)$
if $1/3\leq \mathbb{P}(x\prec y)\leq 2/3$. The 1/3-2/3 Conjecture states that every finite ordered set which is
not totally ordered has a balanced pair. If true, the example (i)
depicted in Figure \ref{alpha} would show that the result is best
possible. The 1/3-2/3 Conjecture first appeared in a paper of
Kislitsyn~\cite{ki}. It was also formulated independently by Fredman
in about 1975 and again by Linial~\cite{li}.

The 1/3-2/3 Conjecture is known to be true for ordered sets with a nontrivial automorphism \cite{ghp}, for ordered sets of width
two \cite{li}, for semiorders \cite{gb1}, for bipartite ordered sets
\cite{tgf}, for 5-thin posets \cite{gb2}, and for 6-thin posets \cite{pm}. See \cite{gb} for a
survey.

In this paper we prove the 1/3-2/3 Conjecture for $N$-free ordered sets.

\begin{figure}[!h]
\begin{center}
\leavevmode \epsfxsize=3.5in \epsfbox{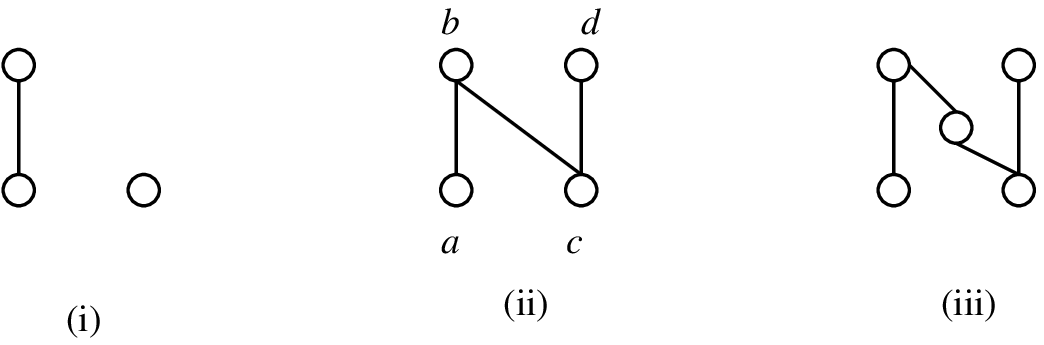}
\end{center}
\caption{} \label{alpha}
\end{figure}


Let $P=(V,\leq)$ be an ordered set. For $x,y\in V$ we say that $y$ is an \emph{upper cover} of $x$ or that $x$ is
\emph{a lower cover} of $y$ if $x<y$ and
there is no element $z\in V$ such that $x<z<y$. Also, we say that
$x$ and $y$ are \emph{comparable} if $x\leq
y$ or $y\leq x$; otherwise we say that $x$ and $y$ are
\emph{incomparable}. A \emph{chain} is a totally ordered set.

A 4-tuple $(a, b, c, d)$ of distinct elements of $V$ is an $N$ in $P$ if $b$ is an upper cover
of $a$ and $c$, $d$ is an upper cover of $c$ and if these are the only comparabilities
between the elements $a,b,c,d$ (See Figure \ref{alpha} (ii)). The ordered set $P$ is $N$-{\it free} if it does  not contain
an $N$ (the ordered set depicted in Figure \ref{alpha} (iii) is $N$-free and the one depicted in Figure \ref{alpha} (ii) is not).

Notice that every finite ordered set can be embedded into a \emph{finite} $N$-free ordered set (see for example \cite{pz}).
It was proved in \cite{blk} that the number of (unlabeled) $N$-free ordered sets is

\[\displaystyle 2^{n\,\log_{2}(n) + o (n\,\log_{2}(n))}.\]

Our main result is this.

\begin{theorem}\label{nfreeinterval} Every finite $N$-free ordered set which is not
totally ordered has a balanced pair.
\end{theorem}

The proof of Theorem \ref{nfreeinterval} is similar to the proof of Theorem 2 of \cite{li} stating that the 1/3-2/3 Conjecture
is true for finite ordered sets of width two (these being the ordered sets covered by two chains).

\section{Proof of Theorem~\ref{nfreeinterval}}

We start this section by stating some useful properties of $N$-free ordered sets.

\begin{lemma}\label{l2} Let $P=(V,\leq)$ be an $N$-free ordered set. If $x, y \in V$ have a common
upper cover, then $x$ and $y$ have the same upper covers. Dually, if
$x, y \in V$ have a common lower cover, then $x$ and $y$ have the
same lower covers.
\end{lemma}

Let $P=(V,\leq)$ be an ordered set. An element $m\in V$ is called \emph{minimal} if for
all $x\in V$ comparable to $m$ we have $x\geq m$. We denote by $Min(P)$ the set
of all minimal elements of $P$. We recall that the decomposition of $P$
into \textit{levels} is the sequence $P_{0},\cdots,P_{l},\cdots$
defined by induction by the formula
\begin{equation*}
P_{l}:=Min(P-\cup \{P_{l^{\prime }}:l^{\prime }<l\}).
\end{equation*}
In particular, $P_{0}=Min(P)$.

\begin{lemma}\label{l2'} Let $P=(V,\leq)$ be an $N$-free ordered set and let $P_0,\cdots,P_h$
be the sequence of its levels. Then for every $x\in V$, there exists $i\leq h$ such that all upper covers of $x$ are in $P_i$.
\end{lemma}
\begin{proof}If $x$ has at most one upper cover, then the conclusion of the lemma holds. So we may assume
that $x$ has at least two distinct upper covers $x_1$ and $x_2$ belonging to two distinct levels. Let $j<k$ be such that
$x_1 \in P_j$ and $x_2 \in P_k$. Then $x_2$ has a lower cover $x_3\in P_{k-1}$. We claim that $(x_3,x_2,x,x_1\}$ is an $N$ in $P$
contradicting our assumption that $P$ is $N$-free. Indeed, since $x_1$ and $x_2$ are upper covers of $x$ we infer that they must be
incomparable.  Moreover, $x_1$ and $x_3$ are incomparable because otherwise $x_1<x_3<x_2$ (notice that $x_3<x_1$ is not
possible since $j\leq k-1$) which contradicts our assumption that $x_2$
is an upper cover of $x$. Similarly we have that $x$ and $x_3$ are incomparable proving our claim. The proof of the lemma is now complete.
\end{proof}

Let $P=(V,\leq)$ be an ordered set. For $x\in V$ define $D(x):=\{y\in V : y< x\}$ and $U(x):=\{y\in V : x<y\}$ .

\begin{lemma}\label{l3} Let $P$ be an $N$-free ordered set and let $P_0,\cdots,P_h$
be the sequence of its levels. Let $0 \leq i\leq h$ be such that $i$ is the largest with the property that $P_i$ contains two distinct elements
with the same set of lower covers. Then for every $x\in P_i$ we have that $U(x)\cup \{x\}$ is a chain.
\end{lemma}
\begin{proof}Let $x\in P_i$ be such that $U(x)\neq \emptyset$ and suppose that $U(x)$ is not a chain.
There is then an element $y\in U(x)\cup \{x\}$ having at least two distinct upper covers, say $y_1,y_2$. From
Lemma~\ref{l2'} we deduce that $y_1$ and $y_2$ are in the same level $P_j$ with $i<j$.
Because $P$ is $N$-free it follows from Lemma \ref{l2} that $y_1$ and $y_2$ have
the same set of lower covers. This contradicts our choice of $i$.
\end{proof}

We recall that an incomparable pair $(x,y)$ of elements is {\it critical } if $U(y)\subseteq
U(x)$ and $D(x)\subseteq D(y)$. The following lemma is true for ordered sets that are not necessarily $N$-free.

\begin{lemma} \label{l1} Suppose $(x,y)$ is a critical pair in $P$
and consider any linear extension of $P$ in which $y<x$. Then the
linear order obtained by swapping the positions of $y$ and $x$ is
also a linear extension of $P$. Moreover, $\mathbb{P}(x\prec y) \geq
\frac{1}{2}$.
\end{lemma}
\begin{proof}Let $L$ be a linear extension that puts $y$ before $x$ and let
$z$ be such that $y\prec z\prec x$ in $L$. Then $z$ is incomparable with both $x$ and $y$ since $(x,y)$ is
a critical pair of $P$. Therefore, the linear order $L'$ obtained by swapping $x$
and $y$ is a linear extension of
$P$. The map $L \mapsto L'$ from the set of linear extensions that
put $y$ before $x$ into the set of linear extensions that put $x$
before $y$ is clearly one-to-one. Hence, $\mathbb{P}(y\prec x) \leq \mathbb{P}(x\prec y)$
and therefore $\mathbb{P}(x\prec y) \geq \frac{1}{2}$.
\end{proof}

We now prove Theorem \ref{nfreeinterval}.


\begin{proof} Let $P=(V,\leq)$ be an $N$-free ordered set not totally ordered and $P_0,\cdots,P_h$
be the sequence of its levels. If $P_0$ is a singleton, say $P_0=\{p_0\}$, then $p_0$ will be the minimum
element in any linear extension of the ordered set. Therefore, nothing will change if $p_0$ is
deleted from the ordered set. So we may assume without loss of generality that $P_0$ has at least two distinct elements.
Notice that any two such elements have the same set of lower covers: the empty set. Now let $0\leq i\leq h$ be such that
$i$ is the largest with the property that $P_i$ contains two distinct elements
with the same set of lower covers and let $a,b\in P_i$ be such elements.
If $U(b)= U(a)=\emptyset$, then $\mathbb{P}(a\prec b)=\frac{1}{2}$ and we are done. Otherwise
we may suppose without loss of generality that $U(b)\neq \emptyset$. From Lemma~\ref{l3} we deduce that $U(b)\cup \{b\}$ is a chain, say
$U(b)\cup \{b\}$ is the chain $b=b_1<\cdots <b_n$. We prove the theorem by contradiction. We may assume without loss of generality that
\[\mathbb{P}(a\prec b_1)<\frac{1}{3}.\]
Indeed, if $U(a)\neq \emptyset$, then the situation is symmetric with respect to
$a$ and $b$ and therefore such an assumption is possible. Otherwise, $U(a)= \emptyset$
and hence $(b_1,a)$ is a critical pair (this is because $D(a)=D(b_1)$ by assumption) yielding
$\mathbb{P}(b_1\prec a)>\frac{2}{3}$ (Lemma~\ref{l1}) or equivalently $\mathbb{P}(a\prec b_1)<\frac{1}{3}$.

Define now the following quantities
\begin{eqnarray*}
q_1 &=& \mathbb{P}(a\prec b_1),\\
q_j &=& \mathbb{P}(b_{j-1}\prec a \prec b_j)(2\leq j\leq n),\\
q_{n+1} &=& \mathbb{P}(b_n\prec a).
\end{eqnarray*}


\noindent\textbf{Lemma.} The real numbers $q_j$ ($1\leq j\leq n+1$) satisfy:
\begin{enumerate}[(i)]
\item $0\leq q_{n+1}\leq \cdots \leq q_1\leq \frac{1}{3},$
\item $\sum_{j=1}^{n+1}q_j=1.$
\end{enumerate}

\begin{proof} Since $q_1,\cdots,q_{n+1}$ is a probability distribution, all we have to show is that
$q_{n+1}\leq \cdots \leq q_1$. To show this we exhibit a one-to-one mapping from the event that $b_j\prec a \prec b_{j+1}$ whose probability
is $q_{j+1}$ into the event that $b_{j-1}\prec a \prec b_{j}$  whose probability is $q_j$ ($1\leq j\leq n$). Notice that in a linear extension for
which $b_j\prec a\prec b_{j+1}$ every element $z$ between $b_j$ and $a$ is incomparable to both $b_j$ and $a$. Indeed, such
an element $z$ cannot be comparable to $b_j$ because otherwise $b_j<z$ in $P$ but the only
element above $b_j$ is $b_{j+1}$ which is above $a$ in the linear extension. Now $z$ cannot be comparable to $a$ as well because otherwise
$z<a$ in $P$ and hence $z<b=b_1<b_j$ (by assumption we have that $D(a)=D(b)$).
The mapping from those linear extensions in which $b_j\prec a\prec b_{j+1}$ to those
in which $b_{j-1}\prec a\prec b_{j}$ is obtained by swapping the positions of $a$ and $b_j$.
This mapping clearly is well defined and one-to-one.
\end{proof}

Theorem~\ref{nfreeinterval} can be proved now: let $r$ be defined by
\[\sum_{j=1}^{r-1}q_j\leq \frac{1}{2}<\sum_{j=1}^{r}q_j\]
Since $\sum_{j=1}^{r-1}q_j=\mathbb{P}(a\prec b_{r-1})\leq \frac{1}{2}$, it follows that $\sum_{j=1}^{r-1}q_j<\frac{1}{3}$.
Similarly $\sum_{j=1}^{r}q_j=\mathbb{P}(a\prec b_{r})$ must be $>\frac{2}{3}$. Therefore $q_r>\frac{1}{3}$, but this contradicts
$\frac{1}{3}>q_1\geq q_r$.
\end{proof}

\end{document}